\documentclass[a4paper,oneside,10pt]{article}%
\usepackage{amsmath}
\usepackage{amsfonts}
\usepackage{amssymb}
\usepackage{graphicx}
\usepackage{color}
\usepackage[square,numbers,sort&compress]{natbib}%
\providecommand{\U}[1]{\protect \rule{.1in}{.1in}}

\newtheorem{theorem}{Theorem}[section]

\newtheorem{definition}[theorem]{Definition}

\newtheorem{example}[theorem]{Example}

\newtheorem{lemma}[theorem]{Lemma}

\newtheorem{proposition}[theorem]{Proposition}
\newtheorem{remark}[theorem]{Remark}

\newenvironment{proof}[1][Proof]{\noindent \textbf{#1.} }{\  $\Box$}
\numberwithin{equation}{section}

\begin{document}

\title{Maximum principle for discrete-time stochastic optimal control problem under
distribution uncertainty}
\author{Mingshang Hu \thanks{Zhongtai Securities Institute for Financial Studies,
Shandong University, Jinan, Shandong 250100, PR China. humingshang@sdu.edu.cn.
Research supported by National Key R\&D Program of China (No. 2018YFA0703900)
and NSF (No. 11671231). }
\and Shaolin Ji\thanks{Zhongtai Securities Institute for Financial Studies,
Shandong University, Jinan, Shandong 250100, PR China. jsl@sdu.edu.cn.
Research supported by NSF (No. 11971263 and 11871458). }
\and Xiaojuan Li\thanks{Zhongtai Securities Institute for Financial Studies,
Shandong University, Jinan 250100, China. Email: lixiaojuan@mail.sdu.edu.cn.} }
\maketitle

\textbf{Abstract}. In this paper, we study a discrete-time stochastic optimal
control problem under distribution uncertainty with convex control domain. By
weak convergence method and Sion's minimax theorem, we obtain the variational
inequality for cost functional under a reference probability $P^{\ast}$.
Moreover, under the square integrability condition for noise and control, we
establish the discrete-time stochastic maximum principle under $P^{\ast}$.
Finally, we introduce a backward algorithm to calculate the reference
probability $P^{\ast}$ and the optimal control $u^{\ast}$.

{\textbf{Key words}. } Stochastic maximum principle, Stochastic optimal
control, Robust control, Sublinear expectation, Volatility uncertainty

\textbf{AMS subject classifications.} 93E20, 60H10, 35K15

\addcontentsline{toc}{section}{\hspace*{1.8em}Abstract}

\section{Introduction}

The stochastic maximum principle is an important tool to solve stochastic
optimal control problems. There are many results on this topic for different
kinds of\ continuous-time stochastic optimal control problems (see \cite{BLM,
YingHu006, YingHu18, Hu17, Hu-JX, Hu-JXu, HP96, LZ, P-90, QT, T, Wu, Yong,
J.Yong} and the reference therein) and discrete-time stochastic optimal
control problems (see \cite{LinZ, WuZ} and the reference therein). In general,
the stochastic optimal control problem is formulated under a given probability
space. But many ecomomic and financial problems involve volatility uncertainty
(see \cite{ALP, DenisMartini2006, EJ-2, EJ-1, L, Peng2004}). In this case, the
stochastic optimal control problem can not be formulated under a given
probability space, because the volatility uncertainty is characterized by a
family of non-dominated probability measures $\mathcal{P}$.

Recently, Hu and Ji \cite{HJ0} studied the stochastic recursive optimal
control problem under volatility uncertainty by using the theory of
$G$-expectation, which was introduced by Peng in \cite{Peng2005, P07a, P08a},
and obtained the related stochastic maximum principle under a reference
probability $P^{\ast}\in \mathcal{P}$.

In this paper, we study the following discrete-time stochastic control system
with convex control domain:%
\[
\left \{
\begin{array}
[c]{rl}%
X_{k+1}= & \displaystyle b\left(  k,X_{k},u_{k}\right)  +\sum_{l=1}^{d}%
\sigma^{l}\left(  k,X_{k},u_{k}\right)  W_{k+1}^{l},\\
X_{0}= & x_{0}\in \mathbb{R}^{n},\text{ }k=0,\ldots,N-1,
\end{array}
\right.
\]
where $N$ is a given and $W_{k}=(W_{k}^{1},\ldots,W_{k}^{d})^{T}$,
$k=1,\ldots,N$, are noises. This kind of control system can be regarded as a
discretization of the control system in \cite{HJ0}, but the difference is that
the noise distribution may be more flexible and not limited to the $G$-normal
distribution. Specifically, as long as a family of probability measures
$\{F_{\theta}:\theta \in \Theta \}$ is given, which characterizes the uncertainty
of the distribution of $W_{k}$, we can construct a sublinear expectation space
$(\Omega,L_{c}^{1}(\mathcal{F}_{N}),\mathbb{\hat{E}})$ (see Preliminaries for
details) to study the above discrete-time stochastic control system. Our
discrete-time stochastic optimal control problem is to minimize the following
cost functional%
\[
J\left(  u\right)  :=\mathbb{\hat{E}}\left[  \sum_{k=0}^{N-1}f\left(
k,X_{k},u_{k}\right)  +\varphi \left(  X_{N}\right)  \right]
\]
over all admissible controls. Due to the representation theorem of sublinear
expectation, $\mathbb{\hat{E}}[\cdot]$ can be represented as an upper
expectation over a family of probability measures $\mathcal{P}$ on
$(\Omega,\mathcal{F}_{N})$. Thus, our discrete-time stochastic optimal control
problem is the robust optimal control problem.

As pointed out in \cite{WuZ}, the integrability of the solution to the adjoint
equation in discrete-time stochastic optimal control problem is completely
different from that in the continuous-time case. In order to make sense of the
adjoint equation to obtain the maximum principle, the integrability of noise
and control is required to depend on $N$ in the literature. However, according
to the actual situation, we expect the integrability requirements of
discrete-time stochastic control system\ to be the same as that of the
continuous-time case. So one purpose of this paper is to make the
integrability of noise and control independent of $N$. On the other hand, the
maximum principle obtained in \cite{HJ0} contains a reference probability
$P^{\ast}$, and there is no general calculation method for this $P^{\ast}$ at
present. So the other purpose of this paper is to give $P^{\ast}$ a better
explanation in discrete-time stochastic optimal control problem, and to give a
method for calculating $P^{\ast}$.

By using the weak convergence method introduced in \cite{HJ0}, we get the
variational equation for cost functional. Furthermore, we obtain the
variational inequality on a reference probability $P^{\ast}$ by Sion's minimax
theorem. In order to derive the maximum principle, we need to consider the
adjoint equation under $P^{\ast}$. Under the square integrability condition
for noise and control, i.e. $|W_{k}|+|u_{k}|\in L_{c}^{2}(\mathcal{F}_{k})$,
we obtain the integrability of the solution to the adjoint equation in Lemma
\ref{le3-5}, which is new in the literature, and then obtain the maximum
principle under $P^{\ast}$. Moreover, we prove that the obtained maximum
principle under $P^{\ast}$ is also a sufficient condition under some convex
assumptions. It is important to emphasize that $P^{\ast}$ is part of the
maximum principle. Thus, the key point to apply maximum principle is to find
$P^{\ast}$. For this purpose, we introduce a backward algorithm to calculate
the reference probability $P^{\ast}$ and the optimal control $u^{\ast}$.

The paper is organized as follows. In Section 2, we give the sublinear
expectation framework of noise distribution and recall some basic results. The
discrete-time stochastic optimal control problem under distribution
uncertainty is formulated in Section 3. In Section 4, we derive the related
discrete-time stochastic maximum principle. The backward algorithm and
examples to apply the obtained maximum principle are given in Section 5.

\section{Preliminaries}

We recall some basic results of sublinear expectations. The readers may refer
to Peng's book \cite{P2019} for more details.

Let $\Omega$ be a given sample space. Throughout the paper, we suppose that
$N$ is a given positive integer. Let $W_{k}=(W_{k}^{1},\ldots,W_{k}^{d}%
)^{T}:\Omega \rightarrow \mathbb{R}^{d}$, $k=1,\ldots,N$, be given functions on
$\Omega$. Consider the following spaces of random variables:%
\[
Lip\left(  \mathcal{F}_{k}\right)  =\{ \phi \left(  W_{1},\cdots,W_{k}\right)
:\forall \phi \in C_{b,Lip}(\mathbb{R}^{d\times k})\},\text{ }k=1,\ldots,N,
\]
where $C_{b,Lip}(\mathbb{R}^{d\times k})$ denotes the space of bounded
Lipschitz functions on $\mathbb{R}^{d\times k}$, $\mathcal{F}_{k}%
=\sigma \{W_{1},\cdots,W_{k}\}$ and $\mathcal{F}_{0}=\{ \emptyset,\Omega \}$. We
now construct a sublinear expectation $\mathbb{\hat{E}}:Lip\left(
\mathcal{F}_{N}\right)  \rightarrow \mathbb{R}$ in the following two steps such
that $W_{k}$, $k=1,\ldots,N$, are independent identically distributed.

Step 1. Denote a family of probability measures on $(\mathbb{R}^{d}%
,\mathcal{B}(\mathbb{R}^{d}))$ by $\{F_{\theta}:\theta \in \Theta \}$, which
characterizes the uncertainty of the distribution of $W_{k}$, $k=1,\ldots,N$.
For each $\varphi \in C_{b,Lip}(\mathbb{R}^{d})$, define%
\[
\mathbb{\hat{E}}\left[  \varphi \left(  W_{k}\right)  \right]  =\sup_{\theta
\in \Theta}\int_{\mathbb{R}^{d}}\varphi(x)F_{\theta}(dx)\text{, }k=1,\ldots,N.
\]

Step 2. For each $X=\phi \left(  W_{1},\cdots,W_{N}\right)  \in Lip\left(
\mathcal{F}_{N}\right)  $, define%
\[
\mathbb{\hat{E}}\left[  X\right]  =\phi_{0},
\]
where $\phi_{0}$ is obtained via the following procedure:%
\[%
\begin{array}
[c]{rcl}%
\phi_{N-1}(x_{1},\ldots,x_{N-1}) & = & \mathbb{\hat{E}}\left[  \phi \left(
x_{1},\ldots,x_{N-1},W_{N}\right)  \right]  ,\\
\phi_{N-2}(x_{1},\ldots,x_{N-2}) & = & \mathbb{\hat{E}}\left[  \phi
_{N-1}\left(  x_{1},\ldots,x_{N-2},W_{N-1}\right)  \right]  ,\\
& \vdots & \\
\phi_{1}(x_{1}) & = & \mathbb{\hat{E}}\left[  \phi_{2}\left(  x_{1}%
,W_{2}\right)  \right]  ,\\
\phi_{0} & = & \mathbb{\hat{E}}\left[  \phi_{1}\left(  W_{1}\right)  \right]
.
\end{array}
\]
The corresponding sublinear expectation $\mathbb{\hat{E}}\left[
X|\mathcal{F}_{k}\right]  :=\phi_{k}(W_{1},\cdots,W_{k})$ for $k=0,\ldots,N$
with $\phi_{N}=\phi$ and $\mathbb{\hat{E}}\left[  X|\mathcal{F}_{0}\right]
=\phi_{0}$.

Denote $W=(W_{1},\ldots,W_{N})$. In order to prove $\mathbb{\hat{E}}\left[
\phi \left(  W\right)  \right]  =\mathbb{\hat{E}}\left[  \tilde{\phi}\left(
W\right)  \right]  $ for $\phi \left(  W\right)  =\tilde{\phi}\left(  W\right)
$ with $\phi$, $\tilde{\phi}\in C_{b,Lip}(\mathbb{R}^{d\times N})$, we need
the following assumption:

\begin{description}
\item[(A1)] $\{W(\omega):\omega \in \Omega \}=A_{1}\times \cdots \times A_{N}$ and
$F_{\theta}\left(  \bar{A}_{k}\right)  =1$ for $\theta \in \Theta$,
$k=1,\ldots,N$, where $A_{k}=\{W_{k}(\omega):\omega \in \Omega \}$, $\bar{A}_{k}$
is the closure of $A_{k}$ for $k=1,\ldots,N$.
\end{description}

Under this assumption, it is easy to check that $\mathbb{\hat{E}}:Lip\left(
\mathcal{F}_{N}\right)  \rightarrow \mathbb{R}$ is well defined. To obtain the
representation of the distribution of $W$, we need the following assumption:

\begin{description}
\item[(A2)] $\sup_{\theta \in \Theta}\int_{\mathbb{R}^{d}}|x|F_{\theta
}(dx)<\infty.$
\end{description}

Set%
\[%
\begin{array}
[c]{rl}%
\mathcal{A}= & \left \{  Q:Q\text{ is a probability measure on }(\mathbb{R}%
^{d\times N},\mathcal{B}(\mathbb{R}^{d\times N}))\text{ such that}\right.
\text{ }\\
& \left.  E_{Q}[\phi]\leq \mathbb{\hat{E}}\left[  \phi \left(  W_{1}%
,\cdots,W_{N}\right)  \right]  \text{ for each }\phi \in C_{b,Lip}%
(\mathbb{R}^{d\times N})\right \}  .
\end{array}
\]
Under the assumptions (A1) and (A2), by Lemma 1.3.5 in \cite{P2019} (see also
Theorem 10 in \cite{HL}), we know that $\mathcal{A}$ is convex, weakly compact
and%
\begin{equation}
\mathbb{\hat{E}}\left[  \phi \left(  W_{1},\cdots,W_{N}\right)  \right]
=\max_{Q\in \mathcal{A}}E_{Q}[\phi]\text{ for each }\phi \in C_{b,Lip}%
(\mathbb{R}^{d\times N}). \label{e-1-1}%
\end{equation}
Here $\mathcal{A}$ characterizes the uncertainty of the distribution of $W$.
In order to give the representation of $\mathbb{\hat{E}}[\cdot]$ on
$(\Omega,\mathcal{F}_{N})$, we need the following assumption:

\begin{description}
\item[(A3)] For each $Q\in \mathcal{A}$, there exists a probability measure $P$
on $(\Omega,\mathcal{F}_{N})$ such that $Q=P\circ W^{-1}$, where
$W=(W_{1},\ldots,W_{N})$.
\end{description}

If $\mathbb{\hat{E}}[\cdot]=E_{P}[\cdot]$ and $W_{k}$, $k=1,\ldots,N$, are
independent under probability measure $P$, then the assumption (A3) holds by
taking $\{F_{\theta}:\theta \in \Theta \}=\{P\circ W_{1}^{-1}\}$. If
$\mathbb{\hat{E}}[\cdot]$ is not a linear expectation, we have the following
sufficient condition for the assumption (A3) to be true.

\begin{lemma}
If the the assumption (A1) holds and $\{W(\omega):\omega \in \Omega \}$ is a
closed set in $\mathbb{R}^{d\times N}$, then the assumption (A3) holds.
\end{lemma}

\begin{proof}
Since $\{W(\omega):\omega \in \Omega \}$ is closed, we know that $A_{k}$,
$k=1,\ldots,N$, in the assumption (A1) are closed. By Tietze's extension
theorem, there exists a sequence $\{ \varphi_{i}^{k}:i\geq1\} \subset
C_{b,Lip}(\mathbb{R}^{d})$ such that $0\leq \varphi_{i}^{k}\leq1$ and $\phi
_{i}^{k}\uparrow I_{A_{k}^{c}}$ for $k=1,\ldots,N$. It is easy to check that
$\mathbb{\hat{E}}\left[  \phi_{i}^{k}\left(  W_{k}\right)  \right]  =0$ for
$i\geq1$ and $k\leq N$. Thus, for each $Q\in \mathcal{A}$, we deduce by
(\ref{e-1-1}) that $Q(\{W(\omega):\omega \in \Omega \})=1$. Define $P$ on
$\mathcal{F}_{N}=W^{-1}(\mathcal{B}(\mathbb{R}^{d\times N}))$ as follows%
\[
P(W^{-1}(A)):=Q(A)=Q(A\cap \{W(\omega):\omega \in \Omega \})\text{ for }%
A\in \mathcal{B}(\mathbb{R}^{d\times N}).
\]
It is easy to verify that $P$ is well defined and $Q=P\circ W^{-1}$, which
implies that the assumption (A3) holds.
\end{proof}

Set%
\[%
\begin{array}
[c]{rl}%
\mathcal{P}= & \left \{  P:P\text{ is a probability measure on }(\Omega
,\mathcal{F}_{N})\text{ such that}\right.  \text{ }\\
& \left.  E_{P}[\phi \left(  W_{1},\cdots,W_{N}\right)  ]\leq \mathbb{\hat{E}%
}\left[  \phi \left(  W_{1},\cdots,W_{N}\right)  \right]  \text{ for each }%
\phi \in C_{b,Lip}(\mathbb{R}^{d\times N})\right \}  .
\end{array}
\]
By the assumption (A3) and (\ref{e-1-1}), we know that
\begin{equation}
\mathcal{A}=\{P\circ W^{-1}:P\in \mathcal{P}\} \label{e-1-2}%
\end{equation}
and%
\begin{equation}
\mathbb{\hat{E}}\left[  \phi \left(  W_{1},\cdots,W_{N}\right)  \right]
=\max_{P\in \mathcal{P}}E_{P}[\phi \left(  W_{1},\cdots,W_{N}\right)  ]\text{
for each }\phi \in C_{b,Lip}(\mathbb{R}^{d\times N}). \label{e-1-3}%
\end{equation}

In the following of this paper, we always suppose that the assumptions (A1),
(A2) and (A3) hold, which implies that (\ref{e-1-1}), (\ref{e-1-2}) and
(\ref{e-1-3}) hold. The capacity associated to $\mathcal{P}$ is defined as
follows%
\[
c(B):=\sup_{P\in \mathcal{P}}P(B)\text{ for }B\in \mathcal{F}_{N}.
\]
A set $B\in \mathcal{F}_{N}$ is polar if $c(B)=0$. A property holds
quasi-surely (q.s. for short) if it holds outside a polar set. In the
following, we do not distinguish two random variables $X$ and $Y$ if $X=Y$
q.s. Set%
\[
\mathbb{L}^{p}(\mathcal{F}_{k}):=\left \{  X\in \mathcal{F}_{k}:\sup
_{P\in \mathcal{P}}E_{P}[|X|^{p}]<\infty \right \}  \text{ for }p\geq1\text{,
}k=1,\ldots,N,
\]
we extend $\mathbb{\hat{E}}[\cdot]$ to $\mathbb{L}^{1}(\mathcal{F}_{N})$
(still denote it by $\mathbb{\hat{E}}[\cdot]$) as follows%
\[
\mathbb{\hat{E}}\left[  X\right]  :=\sup_{P\in \mathcal{P}}E_{P}[X]\text{ for
}X\in \mathbb{L}^{1}(\mathcal{F}_{N}).
\]
By Proposition 14 in \cite{DHP11}, we know that $\mathbb{L}^{p}(\mathcal{F}%
_{k})$ is a Banach space under the norm $||X||_{p}:=(\mathbb{\hat{E}}\left[
|X|^{p}\right]  )^{1/p}$. We denote by $L_{c}^{p}(\mathcal{F}_{k})$ the
completion of $Lip\left(  \mathcal{F}_{k}\right)  $ under the norm
$||X||_{p}:=(\mathbb{\hat{E}}\left[  |X|^{p}\right]  )^{1/p}$ for $p\geq1$ and
$k=1,\ldots,N$. It is clear that $L_{c}^{p}(\mathcal{F}_{k})$ is a closed
subset of $\mathbb{L}^{p}(\mathcal{F}_{k})$. Note that, for $k=1,\ldots,N$ and
$X$, $Y\in Lip\left(  \mathcal{F}_{N}\right)  $,
\[
\mathbb{\hat{E}}\left[  \left \vert \mathbb{\hat{E}}\left[  X|\mathcal{F}%
_{k}\right]  -\mathbb{\hat{E}}\left[  Y|\mathcal{F}_{k}\right]  \right \vert
\right]  \leq \mathbb{\hat{E}}\left[  \left \vert X-Y\right \vert \right]  ,
\]
then $\mathbb{\hat{E}}\left[  \cdot|\mathcal{F}_{k}\right]  $ can be
continuously extended as a mapping $\mathbb{\hat{E}}\left[  \cdot
|\mathcal{F}_{k}\right]  :L_{c}^{1}(\mathcal{F}_{N})\rightarrow L_{c}%
^{1}(\mathcal{F}_{k})$. Similarly, for $k=1,\ldots,N$, we define%
\[
\mathbb{L}^{p}(\mathcal{\tilde{F}}_{k}):=\left \{  \phi \in \mathcal{\tilde{F}%
}_{k}:\sup_{Q\in \mathcal{A}}E_{Q}[|\phi|^{p}]<\infty \right \}  \text{ for
}p\geq1\text{,}%
\]
and denote by $L_{c}^{p}(\mathcal{\tilde{F}}_{k})$ the completion of
$Lip\left(  \mathcal{\tilde{F}}_{k}\right)  =\{ \phi \in C_{b,Lip}%
(\mathbb{R}^{d\times N}):\phi \in \mathcal{\tilde{F}}_{k}\}$ under the norm
$||\phi||_{p}:=(\sup_{Q\in \mathcal{A}}E_{Q}[|\phi|^{p}])^{1/p}$ for $p\geq1$,
where $\mathcal{\tilde{F}}_{k}=\{A\times \mathbb{R}^{d\times(N-k)}%
:A\in \mathcal{B}(\mathbb{R}^{d\times k})\}$. It is well known that for each
$X\in \mathcal{F}_{k}$, there exists a $\phi \in \mathcal{\tilde{F}}_{k}$ such
that $X=\phi(W_{1},\ldots,W_{k})$. Thus, for $p\geq1$, $k=1,\ldots,N$,%
\begin{equation}%
\begin{array}
[c]{rl}%
\mathbb{L}^{p}(\mathcal{F}_{k})= & \left \{  \phi(W_{1},\ldots,W_{k}):\phi
\in \mathbb{L}^{p}(\mathcal{\tilde{F}}_{k})\right \}  ,\\
L_{c}^{p}(\mathcal{F}_{k})= & \left \{  \phi(W_{1},\ldots,W_{k}):\phi \in
L_{c}^{p}(\mathcal{\tilde{F}}_{k})\right \}  .
\end{array}
\label{e-1-4}%
\end{equation}

The following results can be found in Proposition 19 and Lemma 29 in
\cite{DHP11}.

\begin{proposition}
\label{pro1-1}We have

\begin{description}
\item[(i)] Let $X\in L_{c}^{p}(\mathcal{F}_{N})$ be given. Then $\mathbb{\hat
{E}}\left[  |X|^{p}I_{\{|X|\geq i\}}\right]  \downarrow0$ as $i\rightarrow
\infty$.

\item[(ii)] Let $X\in L_{c}^{1}(\mathcal{F}_{N})$ be given. Then for each
$\varepsilon>0$, there exists a $\delta>0$ such that $\mathbb{\hat{E}}\left[
|X|I_{B}\right]  \leq \varepsilon$ for any $B\in \mathcal{F}_{N}$ with
$c(B)\leq \delta$.

\item[(iii)] Let $\{Q^{i}:i\geq1\} \subset \mathcal{A}$ converge weakly to
$Q\in \mathcal{A}$. Then for each $\phi \in L_{c}^{1}(\mathcal{\tilde{F}}_{N})$,
we have $E_{Q^{i}}[\phi]\rightarrow E_{Q}[\phi]$ as $i\rightarrow \infty$.
\end{description}
\end{proposition}

\begin{remark}
By (ii) and the assumption (A3), we can easily deduce that for each $\phi \in
L_{c}^{1}(\mathcal{\tilde{F}}_{N})$ (resp. $X\in L_{c}^{1}(\mathcal{F}_{N})$),
there exists a $Q^{\ast}\in \mathcal{A}$ (resp. $P^{\ast}\in \mathcal{P}$)
satisfying $E_{Q^{\ast}}[\phi]=\sup_{Q\in \mathcal{A}}E_{Q}[\phi]$ (resp.
$E_{P^{\ast}}[X]=\mathbb{\hat{E}}\left[  X\right]  $).
\end{remark}

It is important to note that $W_{k}$, $k=1,\ldots,N$, are not independent for
some $P\in \mathcal{P}$. The following result can be found in Lemma 17 in
\cite{HP21}.

\begin{proposition}
\label{pro1-2}Let $P\in \mathcal{P}$ and $X\in L_{c}^{1}(\mathcal{F}_{N})$ be
given. Then $E_{P}[X|\mathcal{F}_{k}]\leq \mathbb{\hat{E}}\left[
X|\mathcal{F}_{k}\right]  $ $P$-a.s. for $k=0,\ldots,N$.
\end{proposition}

\section{Discrete-time stochastic optimal control problem}

We first give the integrable condition for $W$ and the definition of
admissible controls.

In the following of this paper, we always suppose that $|W_{k}|\in L_{c}%
^{2}(\mathcal{F}_{k})$ for $k=1,\ldots,N$. By (i) of Proposition \ref{pro1-1},
this condition is equivalent to
\[
\sup_{\theta \in \Theta}\int_{\{|x|\geq i\}}|x|^{2}F_{\theta}(dx)\downarrow
0\text{ as }i\rightarrow \infty.
\]

\begin{remark}
\label{re2-1}By Proposition \ref{pro1-2}, we know that for any $P\in
\mathcal{P}$,%
\[
E_{P}\left[  \prod_{k=1}^{N}|W_{k}|^{2}\right]  =E_{P}\left[  E_{P}%
[|W_{N}|^{2}|\mathcal{F}_{N-1}]\prod_{k=1}^{N-1}|W_{k}|^{2}\right]  \leq
E_{P}\left[  \mathbb{\hat{E}}[|W_{N}|^{2}]\prod_{k=1}^{N-1}|W_{k}|^{2}\right]
\leq \prod_{k=1}^{N}\mathbb{\hat{E}}[|W_{k}|^{2}].
\]
Then we can easily obtain $\Pi_{k=1}^{N}|W_{k}|\in L_{c}^{2}(\mathcal{F}_{N}%
)$, which is important to derive the integrable condition for the adjoint equation.
\end{remark}

\begin{definition}
Let $U_{k}\subset \mathbb{R}^{m}$, $k=0,\ldots,N-1$, be nonempty convex sets.
$u=\{u_{k}:$ $k=0$,$\ldots$,$N-1\}$ is said to be an admissible control if
$u_{k}:\Omega \rightarrow U_{k}$ and $|u_{k}|\in L_{c}^{2}(\mathcal{F}_{k})$
for $k=0$,$\ldots$,$N-1$. The set of all admissible controls is denoted by
$\mathcal{U}$.
\end{definition}

Consider the following discrete-time stochastic control system:%
\begin{equation}
\left \{
\begin{array}
[c]{rl}%
X_{k+1}= & \displaystyle b\left(  k,X_{k},u_{k}\right)  +\sum_{l=1}^{d}%
\sigma^{l}\left(  k,X_{k},u_{k}\right)  W_{k+1}^{l},\\
X_{0}= & x_{0}\in \mathbb{R}^{n},\text{ }k=0,\ldots,N-1,
\end{array}
\right.  \label{e2-1}%
\end{equation}
where $b\left(  k,\cdot,\cdot,\cdot \right)  $, $\sigma^{l}\left(
k,\cdot,\cdot,\cdot \right)  :\mathbb{R}^{n}\times U_{k}\times \Omega
\longrightarrow \mathbb{R}^{n}$ for $l=1$,$\ldots$,$d$, $k=0$,$\ldots$,$N-1$.
The cost functional is defined by%
\begin{equation}
J\left(  u\right)  =\mathbb{\hat{E}}\left[  \sum_{k=0}^{N-1}f\left(
k,X_{k},u_{k}\right)  +\varphi \left(  X_{N}\right)  \right]  \text{,}
\label{e2-2}%
\end{equation}
where $f\left(  k,\cdot,\cdot,\cdot \right)  :\mathbb{R}^{n}\times U_{k}%
\times \Omega \longrightarrow \mathbb{R}$ for $k=0$,$\ldots$,$N-1$, and
$\varphi \left(  \cdot,\cdot \right)  :\mathbb{R}^{n}\times \Omega \longrightarrow
\mathbb{R}$.

Our discrete-time stochastic optimal control problem is to minimize the cost
functional $J\left(  u\right)  $ over $\mathcal{U}$, i.e.,%
\begin{equation}
\left \{
\begin{array}
[c]{rl}%
\text{Minimize} & \text{ }J\left(  u\right) \\
\text{subject to} & \text{ }u\in \mathcal{U}.
\end{array}
\right.  \label{e2-3}%
\end{equation}

\section{Discrete-time stochastic maximum principle}

In this section, the constant $C$ will change from line to line in our proof.
In order to derive maximum principle for the discrete-time stochastic optimal
control problem (\ref{e2-3}), we need the following assumptions.

\begin{description}
\item[(H1)] $|b\left(  k,x,u\right)  |$, $|\sigma^{l}\left(  k,x,u\right)
|\in L_{c}^{2}(\mathcal{F}_{k})$, $f\left(  k,x,u\right)  \in L_{c}%
^{1}(\mathcal{F}_{k})$, $\varphi \left(  x\right)  \in L_{c}^{1}(\mathcal{F}%
_{N})$ for $x\in \mathbb{R}^{n}$, $u\in U_{k}$, $l=1$,$\ldots$,$d$,
$k=0$,$\ldots$,$N-1$.

\item[(H2)] There exists a constant $L>0$ such that%
\[
|b_{x}(k,x,u)|+|b_{u}(k,x,u)|+|\sigma_{x}^{l}(k,x,u)|+|\sigma_{u}%
^{l}(k,x,u)|\leq L,
\]%
\[
|f_{x}(k,x,u)|+|f_{u}(k,x,u)|+|\varphi_{x}\left(  x\right)  |\leq
L(1+|x|+|u|),
\]
for $x\in \mathbb{R}^{n}$, $u\in U_{k}$, $l=1$,$\ldots$,$d$, $k=0$,$\ldots
$,$N-1$.

\item[(H3)] For each fixed $\alpha>0$, there exists a modulus of continuity
$\omega_{\alpha}:$ $\left[  0,\infty \right)  \rightarrow \left[  0,\infty
\right)  $ such that for any $x$, $x^{\prime}\in \mathbb{R}^{n}$ with
$|x|\leq \alpha$ and $|x^{\prime}|\leq \alpha$, $u$, $u^{\prime}\in U_{k}$ with
$|u|\leq \alpha$ and $|u^{\prime}|\leq \alpha$,
\[
|\phi(k,x,u)-\phi(k,x^{\prime},u^{\prime})|+|\varphi_{x}(x)-\varphi
_{x}(x^{\prime})|\leq \omega_{\alpha}(|x-x^{\prime}|+|u-u^{\prime}|),
\]
where $\phi(k,\cdot,\cdot)$ is the derivative of $b(k,\cdot,\cdot)$,
$\sigma^{l}(k,\cdot,\cdot)$, $f(k,\cdot,\cdot)$ in $\left(  x,u\right)  $ for
$l=1$,$\ldots$,$d$, $k=0$,$\ldots$,$N-1$.
\end{description}

\begin{lemma}
\label{le3-1}Let assumptions (H1)-(H3) hold. Then for any $u\in \mathcal{U}$,
the equation (\ref{e2-1}) admits a unique solution $\{X_{k}:k\leq N\}$ such
that $|X_{k}|\in L_{c}^{2}(\mathcal{F}_{k})$ for $k\leq N$. Moreover, there
exists a constant $C>0$ depending on $L$, $d$, and $N$ such that, for any
$u\in \mathcal{U}$,%
\begin{equation}
\mathbb{\hat{E}}\left[  \sum_{k=0}^{N}\left \vert X_{k}\right \vert ^{2}\right]
\leq C\left \{  \left \vert x_{0}\right \vert ^{2}+\mathbb{\hat{E}}\left[
\sum_{k=0}^{N-1}\left(  \left \vert u_{k}\right \vert ^{2}+\left \vert b\left(
k,0,0\right)  \right \vert ^{2}+\sum_{l=1}^{d}\left \vert \sigma^{l}\left(
k,0,0\right)  \right \vert ^{2}\right)  \right]  \right \}  . \label{e2-4}%
\end{equation}

\end{lemma}

\begin{proof}
By Remark \ref{re2-1}, it is easy to deduce that the equation (\ref{e2-1}) has
a unique solution $\{X_{k}:k\leq N\}$ such that $|X_{k}|\in L_{c}%
^{2}(\mathcal{F}_{k})$ for $k\leq N$. We only need to prove (\ref{e2-4}).

By Remark \ref{re2-1} and (\ref{e2-1}), we have%
\begin{align*}
\mathbb{\hat{E}}\left[  \left \vert X_{k+1}\right \vert ^{2}\right]   &  \leq
C\left(  \mathbb{\hat{E}}\left[  \left \vert b\left(  k,X_{k},u_{k}\right)
\right \vert ^{2}\right]  +\sum_{l=1}^{d}\mathbb{\hat{E}}\left[  \left \vert
\sigma^{l}\left(  k,X_{k},u_{k}\right)  \right \vert ^{2}\right]
\mathbb{\hat{E}}\left[  \left \vert W_{k+1}^{l}\right \vert ^{2}\right]  \right)
\\
&  \leq C\left(  \mathbb{\hat{E}}\left[  \left \vert X_{k}\right \vert
^{2}\right]  +\mathbb{\hat{E}}\left[  \left \vert u_{k}\right \vert ^{2}\right]
+\mathbb{\hat{E}}\left[  \left \vert b\left(  k,0,0\right)  \right \vert
^{2}\right]  +\sum_{l=1}^{d}\mathbb{\hat{E}}\left[  \left \vert \sigma
^{l}\left(  k,0,0\right)  \right \vert ^{2}\right]  \right)
\end{align*}
for $k=0,\ldots,N-1$, where the constant $C>0$ depending on $L$ and $d$. Then
(\ref{e2-4}) can be obtained by induction.
\end{proof}

\subsection{Variational equation}

Let $u^{\ast}=\{u_{k}^{\ast}:k=0,\ldots,N-1\}$ be optimal and $\{X_{k}^{\ast
}:k=0,\ldots,N\}$ be the corresponding state process of (\ref{e2-1}). Take an
arbitrary $u=\{u_{k}:k=0,\ldots,N-1\} \in \mathcal{U}$. Since $U_{k}$,
$k=0$,$\ldots$,$N-1$, are convex sets, we have $u^{\varepsilon}=u^{\ast
}+\varepsilon \left(  u-u^{\ast}\right)  \in \mathcal{U}$ for $\varepsilon
\in(0,1]$. Similarly, denote $\{X_{k}^{\varepsilon}:k=0,\ldots,N\}$ be the
state process of (\ref{e2-1}) associated with $u^{\varepsilon}$.

The following discrete-time stochastic equation is the variational equation
for (\ref{e2-1}).
\begin{equation}
\left \{
\begin{array}
[c]{rl}%
\hat{X}_{k+1}= & \left[  b_{x}\left(  k\right)  \hat{X}_{k}+b_{u}\left(
k\right)  (u_{k}-u_{k}^{\ast})\right] \\
& \displaystyle+\sum_{l=1}^{d}\left[  \sigma_{x}^{l}\left(  k\right)  \hat
{X}_{k}+\sigma_{u}^{l}\left(  k\right)  (u_{k}-u_{k}^{\ast})\right]
W_{k+1}^{l},\\
\hat{X}_{0}= & 0,\text{ }k=0,\ldots,N-1,
\end{array}
\right.  \label{e2-5}%
\end{equation}
where $b(\cdot)=(b_{1}(\cdot),\ldots,b_{n}(\cdot))^{T}$, $x=(x_{1}%
,\ldots,x_{n})^{T}$,%
\[
b_{x}(\cdot)=\left[
\begin{array}
[c]{ccc}%
b_{1x_{1}}(\cdot) & \cdots & b_{1x_{n}}(\cdot)\\
\vdots &  & \vdots \\
b_{nx_{1}}(\cdot) & \cdots & b_{nx_{n}}(\cdot)
\end{array}
\right]  ,
\]
$b_{x}(k)=b_{x}\left(  k,X_{k}^{\ast},u_{k}^{\ast}\right)  $, similar for
$b_{u}\left(  k\right)  $, $\sigma_{x}^{l}\left(  k\right)  $ and $\sigma
_{u}^{l}\left(  k\right)  $.

Define%
\begin{equation}
\tilde{X}_{k}^{\varepsilon}=\varepsilon^{-1}\left(  X_{k}^{\varepsilon}%
-X_{k}^{\ast}\right)  -\hat{X}_{k}\text{ for }k=0,\ldots,N. \label{e2-6}%
\end{equation}

\begin{proposition}
\label{pro3-2}Let assumptions (H1)-(H3) hold. Then
\begin{equation}
\lim_{\varepsilon \rightarrow0}\sup_{k\leq N}\mathbb{\hat{E}}\left[  |\tilde
{X}_{k}^{\varepsilon}|^{2}\right]  =0. \label{e2-9}%
\end{equation}

\end{proposition}

\begin{proof}
Combining (\ref{e2-1}) and (\ref{e2-5}), we have%
\begin{equation}
\left \{
\begin{array}
[c]{rl}%
\tilde{X}_{k+1}^{\varepsilon}= & \displaystyle \left[  A_{\varepsilon}\left(
k\right)  \tilde{X}_{k}^{\varepsilon}+C_{\varepsilon}\left(  k\right)
\right]  +\sum_{l=1}^{d}\left[  B_{\varepsilon}^{l}\left(  k\right)  \tilde
{X}_{k}^{\varepsilon}+D_{\varepsilon}^{l}\left(  k\right)  \right]
W_{k+1}^{l},\\
\tilde{X}_{0}^{\varepsilon}= & 0,\text{ }k=0,\ldots,N-1,
\end{array}
\right.  \label{e2-7}%
\end{equation}
where%
\[%
\begin{array}
[c]{l}%
A_{\varepsilon}\left(  k\right)  =\int_{0}^{1}b_{x}(k,X_{k}^{\ast}%
+\lambda \varepsilon(\tilde{X}_{k}^{\varepsilon}+\hat{X}_{k}),u_{k}^{\ast
}+\lambda \varepsilon(u_{k}-u_{k}^{\ast}))d\lambda,\\
B_{\varepsilon}^{l}\left(  k\right)  =\int_{0}^{1}\sigma_{x}^{l}(k,X_{k}%
^{\ast}+\lambda \varepsilon(\tilde{X}_{k}^{\varepsilon}+\hat{X}_{k}%
),u_{k}^{\ast}+\lambda \varepsilon(u_{k}-u_{k}^{\ast}))d\lambda,\\
C_{\varepsilon}\left(  k\right)  =\left[  A_{\varepsilon}\left(  k\right)
-b_{x}\left(  k\right)  \right]  \hat{X}_{k}+\int_{0}^{1}[b_{u}(k,X_{k}^{\ast
}+\lambda \varepsilon(\tilde{X}_{k}^{\varepsilon}+\hat{X}_{k}),u_{k}^{\ast
}+\lambda \varepsilon(u_{k}-u_{k}^{\ast}))-b_{u}\left(  k\right)
]d\lambda(u_{k}-u_{k}^{\ast}),\\
D_{\varepsilon}^{l}\left(  k\right)  =\left[  B_{\varepsilon}^{l}\left(
k\right)  -\sigma_{x}^{l}\left(  k\right)  \right]  \hat{X}_{k}+\int_{0}%
^{1}[\sigma_{u}^{l}(k,X_{k}^{\ast}+\lambda \varepsilon(\tilde{X}_{k}%
^{\varepsilon}+\hat{X}_{k}),u_{k}^{\ast}+\lambda \varepsilon(u_{k}-u_{k}^{\ast
}))-\sigma_{u}^{l}\left(  k\right)  ]d\lambda(u_{k}-u_{k}^{\ast}).
\end{array}
\]
By (\ref{e2-4}) in Lemma \ref{le3-1}, we get%
\begin{equation}
\mathbb{\hat{E}}\left[  \sum_{k=0}^{N}|\tilde{X}_{k}^{\varepsilon}%
|^{2}\right]  \leq C\mathbb{\hat{E}}\left[  \sum_{k=0}^{N-1}\left(  \left \vert
C_{\varepsilon}\left(  k\right)  \right \vert ^{2}+\sum_{l=1}^{d}\left \vert
D_{\varepsilon}^{l}\left(  k\right)  \right \vert ^{2}\right)  \right]  ,
\label{e2-8}%
\end{equation}
where the constant $C>0$ depending on $L$, $d$, and $N$. In order to obtain
(\ref{e2-9}), by (\ref{e2-8}) we only need to prove%
\begin{equation}
\mathbb{\hat{E}}\left[  \left \vert C_{\varepsilon}\left(  k\right)
\right \vert ^{2}\right]  +\mathbb{\hat{E}}\left[  \left \vert D_{\varepsilon
}^{l}\left(  k\right)  \right \vert ^{2}\right]  \rightarrow0\text{ as
}\varepsilon \rightarrow0\text{ for }k\leq N-1\text{ and }l\leq d.
\label{e2-10}%
\end{equation}

For each fixed $\alpha>0$ and $k\leq N-1$, set%
\[
S_{1}^{k,\alpha}=\{|X_{k}^{\ast}|\leq \alpha \},\text{ }S_{2}^{k,\alpha
}=\{|\tilde{X}_{k}^{\varepsilon}+\hat{X}_{k}|\leq \alpha \},\text{ }%
S_{3}^{k,\alpha}=\{|u_{k}^{\ast}|\leq \alpha \},\text{ }S_{4}^{k,\alpha
}=\{|u_{k}-u_{k}^{\ast}|\leq \alpha \}.
\]
By (H2) and (H3), we get%
\begin{equation}%
\begin{array}
[c]{rl}%
\left \vert C_{\varepsilon}\left(  k\right)  \right \vert  & =\left \vert
C_{\varepsilon}\left(  k\right)  \right \vert \left(  I_{\cap_{i=1}^{4}%
S_{i}^{k,\alpha}}+I_{(\cap_{i=1}^{4}S_{i}^{k,\alpha})^{c}}\right) \\
& \leq \omega_{2\alpha}(2\varepsilon \alpha)(|\hat{X}_{k}|+|u_{k}-u_{k}^{\ast
}|)+2L(|\hat{X}_{k}|+|u_{k}-u_{k}^{\ast}|)I_{(\cap_{i=1}^{4}S_{i}^{k,\alpha
})^{c}}.
\end{array}
\label{e2-14}%
\end{equation}
Noting that $\omega_{2\alpha}(2\varepsilon \alpha)\rightarrow0$ as
$\varepsilon \rightarrow0$, then we obtain
\begin{equation}
\underset{\varepsilon \rightarrow0}{\lim \sup}\mathbb{\hat{E}}\left[  \left \vert
C_{\varepsilon}\left(  k\right)  \right \vert ^{2}\right]  \leq C\mathbb{\hat
{E}}\left[  (|\hat{X}_{k}|+|u_{k}-u_{k}^{\ast}|)^{2}I_{(\cap_{i=1}^{4}%
S_{i}^{k,\alpha})^{c}}\right]  , \label{e2-15}%
\end{equation}
where the constant $C>0$ depending on $L$.

Since%
\[
I_{(\cap_{i=1}^{4}S_{i}^{k,\alpha})^{c}}\leq \alpha^{-2}(|X_{k}^{\ast}%
|^{2}+|\tilde{X}_{k}^{\varepsilon}+\hat{X}_{k}|^{2}+|u_{k}^{\ast}|^{2}%
+|u_{k}-u_{k}^{\ast}|^{2}),
\]
we deduce%
\begin{equation}
c((\cap_{i=1}^{4}S_{i}^{k,\alpha})^{c})=\mathbb{\hat{E}}\left[  I_{(\cap
_{i=1}^{4}S_{i}^{k,\alpha})^{c}}\right]  \leq \alpha^{-2}\mathbb{\hat{E}%
}\left[  |X_{k}^{\ast}|^{2}+|\tilde{X}_{k}^{\varepsilon}+\hat{X}_{k}%
|^{2}+|u_{k}^{\ast}|^{2}+|u_{k}-u_{k}^{\ast}|^{2}\right]  . \label{e2-16}%
\end{equation}
It follows from (H2) and (\ref{e2-8}) that%
\begin{equation}
\mathbb{\hat{E}}\left[  \sum_{k=0}^{N}|\tilde{X}_{k}^{\varepsilon}%
|^{2}\right]  \leq C\mathbb{\hat{E}}\left[  \sum_{k=0}^{N-1}\left(  \left \vert
\hat{X}_{k}\right \vert ^{2}+\left \vert u_{k}-u_{k}^{\ast}\right \vert
^{2}\right)  \right]  , \label{e2-11}%
\end{equation}
where the constant $C>0$ depending on $L$, $d$, and $N$. By (\ref{e2-16}) and
(\ref{e2-11}), we have%
\begin{equation}
\lim_{\alpha \rightarrow \infty}c((\cap_{i=1}^{4}S_{i}^{k,\alpha})^{c})=0.
\label{e2-12}%
\end{equation}
Thus, by (\ref{e2-15}), (\ref{e2-12}) and (ii) in Proposition \ref{pro1-1}, we
obtain%
\[
\lim_{\varepsilon \rightarrow0}\mathbb{\hat{E}}\left[  \left \vert
C_{\varepsilon}\left(  k\right)  \right \vert ^{2}\right]  =0.
\]
The same analysis for $D_{\varepsilon}^{l}\left(  k\right)  $, we can get
(\ref{e2-10}).
\end{proof}

Now we consider the variation for cost functional (\ref{e2-2}).

Set%
\begin{equation}
\mathcal{P}^{\ast}=\left \{  P\in \mathcal{P}:E_{P}\left[  \sum_{k=0}%
^{N-1}f\left(  k,X_{k}^{\ast},u_{k}^{\ast}\right)  +\varphi \left(  X_{N}%
^{\ast}\right)  \right]  =\mathbb{\hat{E}}\left[  \sum_{k=0}^{N-1}f\left(
k,X_{k}^{\ast},u_{k}^{\ast}\right)  +\varphi \left(  X_{N}^{\ast}\right)
\right]  \right \}  , \label{e2-17}%
\end{equation}%
\begin{equation}
\mathcal{P}^{\varepsilon}=\left \{  P\in \mathcal{P}:E_{P}\left[  \sum
_{k=0}^{N-1}f\left(  k,X_{k}^{\varepsilon},u_{k}^{\varepsilon}\right)
+\varphi \left(  X_{N}^{\varepsilon}\right)  \right]  =\mathbb{\hat{E}}\left[
\sum_{k=0}^{N-1}f\left(  k,X_{k}^{\varepsilon},u_{k}^{\varepsilon}\right)
+\varphi \left(  X_{N}^{\varepsilon}\right)  \right]  \right \}  \label{e2-19}%
\end{equation}
and%
\begin{equation}
\Theta^{u}=\sum_{k=0}^{N-1}\left[  f_{x}(k)\hat{X}_{k}+f_{u}\left(  k\right)
(u_{k}-u_{k}^{\ast})\right]  +\varphi_{x}\left(  X_{N}^{\ast}\right)  \hat
{X}_{N}, \label{e2-18}%
\end{equation}
where $f_{x}(\cdot)=(f_{x_{1}}(\cdot),\ldots,f_{x_{n}}(\cdot))$,
$f_{x}(k)=f_{x}\left(  k,X_{k}^{\ast},u_{k}^{\ast}\right)  $, similar for
$f_{u}(\cdot)$, $\varphi_{x}\left(  \cdot \right)  $, $f_{u}\left(  k\right)  $
and $f(k)$.

\begin{theorem}
\label{th3-3}Let assumptions (H1)-(H3) hold. Then, for any $u\in \mathcal{U}$,
there exists a $P^{u}\in \mathcal{P}^{\ast}$ such that%
\begin{equation}
\underset{\varepsilon \rightarrow0}{\lim}\varepsilon^{-1}[J\left(
u^{\varepsilon}\right)  -J\left(  u^{\ast}\right)  ]=E_{P^{u}}\left[
\Theta^{u}\right]  =\underset{P\in \mathcal{P}^{\ast}}{\sup}E_{P}\left[
\Theta^{u}\right]  . \label{e2-20}%
\end{equation}

\end{theorem}

\begin{proof}
Set%
\[
I_{\varepsilon}=\varepsilon^{-1}\left[  \sum_{k=0}^{N-1}(f\left(
k,X_{k}^{\varepsilon},u_{k}^{\varepsilon}\right)  -f\left(  k,X_{k}^{\ast
},u_{k}^{\ast}\right)  )+\varphi \left(  X_{N}^{\varepsilon}\right)
-\varphi \left(  X_{N}^{\ast}\right)  \right]  .
\]
Similar to (\ref{e2-7}), we have%
\[
I_{\varepsilon}-\Theta^{u}=\sum_{k=0}^{N-1}[L_{\varepsilon}\left(  k\right)
\tilde{X}_{k}^{\varepsilon}+G_{\varepsilon}\left(  k\right)  ]+H_{\varepsilon
},
\]
where $\tilde{X}_{k}^{\varepsilon}$ is defined in (\ref{e2-6}),%
\[%
\begin{array}
[c]{rl}%
L_{\varepsilon}\left(  k\right)  = & \int_{0}^{1}f_{x}(k,X_{k}^{\ast}%
+\lambda \varepsilon(\tilde{X}_{k}^{\varepsilon}+\hat{X}_{k}),u_{k}^{\ast
}+\lambda \varepsilon(u_{k}-u_{k}^{\ast}))d\lambda,\\
G_{\varepsilon}\left(  k\right)  = & \left[  L_{\varepsilon}\left(  k\right)
-f_{x}\left(  k\right)  \right]  \hat{X}_{k}+\int_{0}^{1}[f_{u}(k,X_{k}^{\ast
}+\lambda \varepsilon(\tilde{X}_{k}^{\varepsilon}+\hat{X}_{k}),u_{k}^{\ast
}+\lambda \varepsilon(u_{k}-u_{k}^{\ast}))-f_{u}\left(  k\right)
]d\lambda(u_{k}-u_{k}^{\ast}),\\
H_{\varepsilon}= & \int_{0}^{1}\varphi_{x}(X_{N}^{\ast}+\lambda \varepsilon
(\tilde{X}_{N}^{\varepsilon}+\hat{X}_{N}))d\lambda \tilde{X}_{N}^{\varepsilon
}+\int_{0}^{1}[\varphi_{x}(X_{N}^{\ast}+\lambda \varepsilon(\tilde{X}%
_{N}^{\varepsilon}+\hat{X}_{N}))-\varphi_{x}(X_{N}^{\ast})]d\lambda \hat{X}%
_{N}.
\end{array}
\]
It follows from (H2), (\ref{e2-9}) and H\"{o}lder's inequality that%
\[
\lim_{\varepsilon \rightarrow0}\mathbb{\hat{E}}\left[  \sum_{k=0}%
^{N-1}|L_{\varepsilon}\left(  k\right)  \tilde{X}_{k}^{\varepsilon
}|+\left \vert \int_{0}^{1}\varphi_{x}(X_{N}^{\ast}+\lambda \varepsilon
(\tilde{X}_{N}^{\varepsilon}+\hat{X}_{N}))d\lambda \tilde{X}_{N}^{\varepsilon
}\right \vert \right]  =0.
\]
Similar to the proof of (\ref{e2-15}) and%
\[
\lim_{\varepsilon \rightarrow0}\mathbb{\hat{E}}\left[  |\tilde{X}%
_{k}^{\varepsilon}|(|\hat{X}_{k}|+|u_{k}-u_{k}^{\ast}|)\right]  =0,
\]
we deduce
\[
\underset{\varepsilon \rightarrow0}{\lim \sup}\mathbb{\hat{E}}\left[  \left \vert
G_{\varepsilon}\left(  k\right)  \right \vert \right]  \leq C\mathbb{\hat{E}%
}\left[  (1+|X_{k}^{\ast}|+|\hat{X}_{k}|+|u_{k}^{\ast}|+|u_{k}|)(|\hat{X}%
_{k}|+|u_{k}-u_{k}^{\ast}|)I_{(\cap_{i=1}^{4}S_{i}^{k,\alpha})^{c}}\right]  ,
\]
where the constant $C>0$ depending on $L$, $S_{i}^{k,\alpha}$, $i\leq4$, are
defined in the proof in Proposition \ref{pro3-2}. By (\ref{e2-12}) and (ii) in
Proposition \ref{pro1-1}, we have%
\[
\lim_{\varepsilon \rightarrow0}\mathbb{\hat{E}}\left[  \left \vert
G_{\varepsilon}\left(  k\right)  \right \vert \right]  =0.
\]
Similarly, we can get%
\[
\lim_{\varepsilon \rightarrow0}\mathbb{\hat{E}}\left[  \left \vert \int_{0}%
^{1}[\varphi_{x}(X_{N}^{\ast}+\lambda \varepsilon(\tilde{X}_{N}^{\varepsilon
}+\hat{X}_{N}))-\varphi_{x}(X_{N}^{\ast})]d\lambda \hat{X}_{N}\right \vert
\right]  =0.
\]
Thus we obtain%
\begin{equation}
\lim_{\varepsilon \rightarrow0}\mathbb{\hat{E}}\left[  \left \vert
I_{\varepsilon}-\Theta^{u}\right \vert \right]  =0. \label{e2-22}%
\end{equation}

For each $P\in \mathcal{P}^{\ast}$, we have%
\[
\mathbb{\hat{E}}\left[  \sum_{k=0}^{N-1}f\left(  k,X_{k}^{\varepsilon}%
,u_{k}^{\varepsilon}\right)  +\varphi \left(  X_{N}^{\varepsilon}\right)
\right]  \geq E_{P}\left[  \sum_{k=0}^{N-1}f\left(  k,X_{k}^{\varepsilon
},u_{k}^{\varepsilon}\right)  +\varphi \left(  X_{N}^{\varepsilon}\right)
\right]  ,
\]
which implies%
\[
\varepsilon^{-1}[J\left(  u^{\varepsilon}\right)  -J\left(  u^{\ast}\right)
]\geq E_{P}\left[  I_{\varepsilon}\right]  .
\]
By (\ref{e2-22}), we know%
\[
|E_{P}\left[  I_{\varepsilon}\right]  -E_{P}\left[  \Theta^{u}\right]
|\leq \mathbb{\hat{E}}\left[  \left \vert I_{\varepsilon}-\Theta^{u}\right \vert
\right]  \rightarrow0\text{ as }\varepsilon \rightarrow0,
\]
which implies%
\[
\underset{\varepsilon \rightarrow0}{\lim \inf}\varepsilon^{-1}[J\left(
u^{\varepsilon}\right)  -J\left(  u^{\ast}\right)  ]\geq E_{P}\left[
\Theta^{u}\right]  \text{ for each }P\in \mathcal{P}^{\ast}.
\]
Thus%
\begin{equation}
\underset{\varepsilon \rightarrow0}{\lim \inf}\varepsilon^{-1}[J\left(
u^{\varepsilon}\right)  -J\left(  u^{\ast}\right)  ]\geq \underset
{P\in \mathcal{P}^{\ast}}{\sup}E_{P}\left[  \Theta^{u}\right]  . \label{e2-23}%
\end{equation}

Let $\{ \varepsilon_{i}:i\geq1\} \subset(0,1]$ satisfy $\varepsilon
_{i}\rightarrow0$ and
\[
\underset{\varepsilon \rightarrow0}{\lim \sup}\varepsilon^{-1}[J\left(
u^{\varepsilon}\right)  -J\left(  u^{\ast}\right)  ]=\lim_{i\rightarrow \infty
}\varepsilon_{i}^{-1}[J\left(  u^{\varepsilon_{i}}\right)  -J\left(  u^{\ast
}\right)  ].
\]
For each $i\geq1$, choose a $P^{i}\in \mathcal{P}^{\varepsilon_{i}}$. Since%
\[
\mathbb{\hat{E}}\left[  \sum_{k=0}^{N-1}f\left(  k,X_{k}^{\ast},u_{k}^{\ast
}\right)  +\varphi \left(  X_{N}^{\ast}\right)  \right]  \geq E_{P^{i}}\left[
\sum_{k=0}^{N-1}f\left(  k,X_{k}^{\ast},u_{k}^{\ast}\right)  +\varphi \left(
X_{N}^{\ast}\right)  \right]  ,
\]
we have%
\[
\varepsilon_{i}^{-1}[J\left(  u^{\varepsilon_{i}}\right)  -J\left(  u^{\ast
}\right)  ]\leq E_{P^{i}}\left[  I_{\varepsilon_{i}}\right]  .
\]
Due to (\ref{e2-22}), we know%
\[
|E_{P^{i}}\left[  I_{\varepsilon_{i}}\right]  -E_{P^{i}}\left[  \Theta
^{u}\right]  |\leq \mathbb{\hat{E}}\left[  \left \vert I_{\varepsilon_{i}%
}-\Theta^{u}\right \vert \right]  \rightarrow0\text{ as }i\rightarrow \infty.
\]
Thus we obtain%
\begin{equation}
\underset{\varepsilon \rightarrow0}{\lim \sup}\varepsilon^{-1}[J\left(
u^{\varepsilon}\right)  -J\left(  u^{\ast}\right)  ]\leq \underset
{i\rightarrow \infty}{\lim \inf}E_{P^{i}}\left[  \Theta^{u}\right]  .
\label{e2-24}%
\end{equation}

By (\ref{e-1-4}), there exist $\phi_{i}\in L_{c}^{1}(\mathcal{\tilde{F}}_{N}%
)$, $i\geq1$, and $\phi$, $\phi^{u}\in L_{c}^{1}(\mathcal{\tilde{F}}_{N})$
such that%
\begin{equation}
\sum_{k=0}^{N-1}f\left(  k,X_{k}^{\varepsilon_{i}},u_{k}^{\varepsilon_{i}%
}\right)  +\varphi \left(  X_{N}^{\varepsilon_{i}}\right)  =\phi_{i}(W)\text{,
}\sum_{k=0}^{N-1}f\left(  k,X_{k}^{\ast},u_{k}^{\ast}\right)  +\varphi \left(
X_{N}^{\ast}\right)  =\phi(W),\text{ }\Theta^{u}=\phi^{u}(W), \label{new-e2-1}%
\end{equation}
where $W=(W_{1},\ldots,W_{N})$. Set $Q^{i}=P^{i}\circ W^{-1}$ for $i\geq1$, we
have%
\begin{equation}
J\left(  u^{\varepsilon_{i}}\right)  =E_{Q^{i}}[\phi_{i}]\text{ and }E_{P^{i}%
}\left[  \Theta^{u}\right]  =E_{Q^{i}}[\phi^{u}]. \label{e2-26}%
\end{equation}
Since $\mathcal{A}$ is weakly compact, we can find a subsequence $\{Q^{i_{j}%
}:j\geq1\}$ of $\{Q^{i}:i\geq1\}$ such that $Q^{i_{j}}$ converge weakly to
$Q^{u}\in \mathcal{A}$. Let $P^{u}\in \mathcal{P}$ satisfy $Q^{u}=P^{u}\circ
W^{-1}$. Then, by (iii) of Proposition \ref{pro1-1}, we obtain%
\begin{equation}
\underset{i\rightarrow \infty}{\lim \inf}E_{P^{i}}\left[  \Theta^{u}\right]
\leq \lim_{j\rightarrow \infty}E_{Q^{i_{j}}}[\phi^{u}]=E_{Q^{u}}[\phi
^{u}]=E_{P^{u}}[\Theta^{u}]. \label{e2-27}%
\end{equation}
Now we prove $P^{u}\in \mathcal{P}^{\ast}$. Since%
\[
|J\left(  u^{\varepsilon_{i}}\right)  -J\left(  u^{\ast}\right)
|\leq \varepsilon_{i}\mathbb{\hat{E}}\left[  |I_{\varepsilon_{i}}|\right]
\text{ and }E_{Q^{i}}[|\phi_{i}-\phi|]=\varepsilon_{i}E_{P^{i}}%
[|I_{\varepsilon_{i}}|]\leq \varepsilon_{i}\mathbb{\hat{E}}\left[
|I_{\varepsilon_{i}}|\right]  ,
\]
we deduce by (\ref{e2-22}), (\ref{e2-26}) and (iii) of Proposition
\ref{pro1-1} that%
\[
J\left(  u^{\ast}\right)  =\lim_{i\rightarrow \infty}J\left(  u^{\varepsilon
_{i}}\right)  =\lim_{j\rightarrow \infty}E_{Q^{i_{j}}}[\phi]=E_{Q^{u}}%
[\phi]=E_{P^{u}}\left[  \sum_{k=0}^{N-1}f\left(  k,X_{k}^{\ast},u_{k}^{\ast
}\right)  +\varphi \left(  X_{N}^{\ast}\right)  \right]  ,
\]
which implies $P^{u}\in \mathcal{P}^{\ast}$. Thus, by (\ref{e2-23}),
(\ref{e2-24}) and (\ref{e2-27}), we obtain (\ref{e2-20}).
\end{proof}

\subsection{ Variational inequality.}

\begin{theorem}
\label{th3-4}Let assumptions (H1)-(H3) hold. Then, there exists a $P^{\ast}%
\in \mathcal{P}^{\ast}$ such that%
\[
\inf_{u\in \mathcal{U}}E_{P^{\ast}}[\Theta^{u}]\geq0.
\]

\end{theorem}

\begin{proof}
Set
\begin{equation}
\mathcal{A}^{\ast}=\{P\circ W^{-1}:P\in \mathcal{P}^{\ast}\}. \label{e2-29}%
\end{equation}
It is easy to check that%
\begin{equation}
\mathcal{A}^{\ast}=\left \{  Q\in \mathcal{A}:E_{Q}[\phi]=\sup_{Q^{\prime}%
\in \mathcal{A}}E_{Q^{\prime}}[\phi]\right \}  , \label{e2-28}%
\end{equation}
where $\phi$, $\phi^{u}\in L_{c}^{1}(\mathcal{\tilde{F}}_{N})$ are defined in
(\ref{new-e2-1}). By (iii) of Proposition \ref{pro1-1}, we deduce that
$\mathcal{A}^{\ast}$ is weakly compact. It follows from Theorem (\ref{th3-3})
that%
\[
\inf_{u\in \mathcal{U}}\sup_{P\in \mathcal{P}^{\ast}}E_{P}[\Theta^{u}%
]=\inf_{u\in \mathcal{U}}\sup_{Q\in \mathcal{A}^{\ast}}E_{Q}[\phi^{u}]\geq0.
\]
It is easy to verify that
\[
\phi^{\lambda u+(1-\lambda)u^{\prime}}=\lambda \phi^{u}+(1-\lambda
)\phi^{u^{\prime}}\text{ for }\lambda \in \lbrack0,1]\text{, }u,\text{
}u^{\prime}\in \mathcal{U}.
\]
Then, by Sion's minimax theorem, we obtain%
\[
\sup_{Q\in \mathcal{A}^{\ast}}\inf_{u\in \mathcal{U}}E_{Q}[\phi^{u}]=\inf
_{u\in \mathcal{U}}\sup_{Q\in \mathcal{A}^{\ast}}E_{Q}[\phi^{u}]\geq0.
\]
For each $i\geq1$, there exists a $Q^{i}\in \mathcal{A}^{\ast}$ such that%
\[
\inf_{u\in \mathcal{U}}E_{Q^{i}}[\phi^{u}]\geq-\frac{1}{i}.
\]
Since $\mathcal{A}^{\ast}$ is weakly compact, we can find a subsequence
$\{Q^{i_{j}}:j\geq1\}$ of $\{Q^{i}:i\geq1\}$ such that $Q^{i_{j}}$ converge
weakly to $Q^{\ast}\in \mathcal{A}^{\ast}$. Then, for each $u\in \mathcal{U}$,
we have%
\[
E_{Q^{\ast}}[\phi^{u}]=\lim_{j\rightarrow \infty}E_{Q^{i_{j}}}[\phi^{u}]\geq0.
\]
By (\ref{e2-29}), we know that there exists a $P^{\ast}\in \mathcal{P}^{\ast}$
such that $Q^{\ast}=P^{\ast}\circ W^{-1}$. Thus we obtain%
\[
\inf_{u\in \mathcal{U}}E_{P^{\ast}}[\Theta^{u}]=\inf_{u\in \mathcal{U}%
}E_{Q^{\ast}}[\phi^{u}]\geq0.
\]

\end{proof}

\subsection{Maximum principle}

Consider the following adjoint equation under $P^{\ast}$:%
\begin{equation}
\left \{
\begin{array}
[c]{rl}%
P_{N-1}= & E_{P^{\ast}}[\varphi_{x}^{T}(X_{N}^{\ast})|\mathcal{F}%
_{N-1}],\text{ }Q_{N-1}=E_{P^{\ast}}[\varphi_{x}^{T}(X_{N}^{\ast})W_{N}%
^{T}|\mathcal{F}_{N-1}],\\
P_{k}= & \displaystyle E_{P^{\ast}}\left[  b_{x}^{T}(k+1)P_{k+1}+\sum
_{l=1}^{d}(\sigma_{x}^{l}(k+1))^{T}Q_{k+1}^{l}+f_{x}^{T}%
(k+1)\Big{|}\mathcal{F}_{k}\right]  ,\\
Q_{k}= & \displaystyle E_{P^{\ast}}\left[  \left(  b_{x}^{T}(k+1)P_{k+1}%
+\sum_{l=1}^{d}(\sigma_{x}^{l}(k+1))^{T}Q_{k+1}^{l}+f_{x}^{T}(k+1)\right)
W_{k+1}^{T}\Big{|}\mathcal{F}_{k}\right]  ,\\
k= & N-2,\ldots,0,
\end{array}
\right.  \label{e2-30}%
\end{equation}
where $P_{k}:\Omega \rightarrow \mathbb{R}^{n}$, $Q_{k}=[Q_{k}^{1},\ldots
,Q_{k}^{d}]:\Omega \rightarrow \mathbb{R}^{n\times d}$.

\begin{lemma}
\label{le3-5}Let assumptions (H1)-(H3) hold. Then the adjoint equation
(\ref{e2-30}) has a unique solution $\{(P_{k},Q_{k}):k=0$,$\ldots$,$N-1\}$
such that, for each $\xi \in L_{P^{\ast}}^{2}(\mathcal{F}_{k})$,
\begin{equation}
E_{P^{\ast}}\left[  |P_{k}||\xi|+|Q_{k}||\xi|\right]  <\infty \text{ for
}k=N-1,\ldots,0, \label{e2-31}%
\end{equation}
where $L_{P^{\ast}}^{2}(\mathcal{F}_{k})=\{ \xi \in \mathcal{F}_{k}:E_{P^{\ast}%
}[|\xi|^{2}]<\infty \}$.
\end{lemma}

\begin{proof}
Set $A_{k}=b_{x}(k)+\sum_{l=1}^{d}\sigma_{x}^{l}(k)W_{k+1}^{l}$, we prove
that, for $k=N-1$,$\ldots$,$0$,
\begin{equation}
\left \{
\begin{array}
[c]{rl}%
P_{k}= & \displaystyle E_{P^{\ast}}\left[  f_{x}^{T}(k+1)+\sum_{i=k+2}%
^{N}\left(  \prod \limits_{j=k+1}^{i-1}A_{j}^{T}\right)  f_{x}^{T}%
(i)\Big{|}\mathcal{F}_{k}\right]  ,\\
Q_{k}= & \displaystyle E_{P^{\ast}}\left[  f_{x}^{T}(k+1)W_{k+1}^{T}%
+\sum_{i=k+2}^{N}\left(  \prod \limits_{j=k+1}^{i-1}A_{j}^{T}\right)  f_{x}%
^{T}(i)W_{k+1}^{T}\Big{|}\mathcal{F}_{k}\right]  ,
\end{array}
\right.  \label{e2-32}%
\end{equation}
where $f_{x}(N)=\varphi_{x}(X_{N}^{\ast})$ and $\sum_{i=N+1}^{N}[\cdot]=0$. By
assumption (H2), it is easy to verify that%
\[
E_{P^{\ast}}[|f_{x}^{T}(i)|^{2}]<\infty \text{ and }|W_{k+1}^{T}|\prod
\limits_{j=k+1}^{i-1}|A_{j}^{T}|\leq C\prod \limits_{j=k+1}^{i}(1+|W_{j}|),
\]
where the constant $C>0$ depending on $L$, $d$, and $N$. By Remark
\ref{re2-1}, we know%
\[
E_{P^{\ast}}\left[  |\xi|^{2}|W_{k+1}^{T}|^{2}\prod \limits_{j=k+1}^{i-1}%
|A_{j}^{T}|^{2}\right]  \leq CE_{P^{\ast}}[|\xi|^{2}]\prod \limits_{j=k+1}%
^{i}\mathbb{\hat{E}}[1+|W_{j}|^{2}]<\infty.
\]
Thus $E_{P^{\ast}}\left[  |P_{k}||\xi|+|Q_{k}||\xi|\right]  <\infty$ if
(\ref{e2-32}) holds.

For $k=N-1$, it is easy to verify that (\ref{e2-32}) holds. If (\ref{e2-32})
holds for $k\geq1$, then, by (\ref{e2-30}), it is easy to verify that
(\ref{e2-32}) holds for $k-1$. Thus (\ref{e2-32}) holds by induction.
\end{proof}

For $k=0,\ldots,N-1$, define the Hamiltonian function $H\left(  k,\cdot
,\cdot,\cdot,\cdot,\cdot \right)  :\mathbb{R}^{n}\times U_{k}\times
\mathbb{R}^{n}\times \mathbb{R}^{n\times d}\times \Omega \longrightarrow
\mathbb{R}$ as follows:%
\begin{equation}
H\left(  k,x,u,p,q\right)  =p^{T}b(k,x,u)+\sum_{l=1}^{d}(q^{l})^{T}\sigma
^{l}(k,x,u)+f(k,x,u), \label{e2-33}%
\end{equation}
where $q=[q^{1},\ldots,q^{d}]$. It is easy to check that%
\[
H_{u}\left(  k,x,u,p,q\right)  =p^{T}b_{u}(k,x,u)+\sum_{l=1}^{d}(q^{l}%
)^{T}\sigma_{u}^{l}(k,x,u)+f_{u}(k,x,u).
\]

Now we give the following stochastic maximum principle.

\begin{theorem}
\label{th3-6}Suppose assumptions (H1)-(H3) hold. Let $u^{\ast}=\{u_{k}^{\ast
}:k=0,\ldots,N-1\}$ be optimal and $\{X_{k}^{\ast}:k=0,\ldots,N\}$ be the
corresponding state process of (\ref{e2-1}). Then there exist a $P^{\ast}%
\in \mathcal{P}^{\ast}$ such that%
\begin{equation}
H_{u}\left(  k,X_{k}^{\ast},u_{k}^{\ast},P_{k},Q_{k}\right)  (u-u_{k}^{\ast
})\geq0,\text{ }P^{\ast}\text{-a.s.},\text{ }\forall u\in U_{k},\text{
}k=0,\ldots,N-1, \label{e2-34}%
\end{equation}
where $\{(P_{k},Q_{k}):k=0$,$\ldots$,$N-1\}$ is the solution of the adjoint
equation (\ref{e2-30}) under $P^{\ast}$.
\end{theorem}

\begin{proof}
By Theorem \ref{th3-4}, there exists a $P^{\ast}\in \mathcal{P}^{\ast}$ such
that%
\[
E_{P^{\ast}}\left[  \sum_{k=0}^{N-1}f_{x}(k)\hat{X}_{k}+\varphi_{x}\left(
X_{N}^{\ast}\right)  \hat{X}_{N}+\sum_{k=0}^{N-1}f_{u}\left(  k\right)
(u_{k}-u_{k}^{\ast})\right]  \geq0\text{ for any }u\in \mathcal{U}.
\]
Let $\{(P_{k},Q_{k}):k=0$,$\ldots$,$N-1\}$ be the solution of the adjoint
equation (\ref{e2-30}) under $P^{\ast}$. By (\ref{e2-5}), (\ref{e2-30}) and
(\ref{e2-31}), it is easy to verify that%
\begin{align*}
E_{P^{\ast}}\left[  \varphi_{x}\left(  X_{N}^{\ast}\right)  \hat{X}%
_{N}|\mathcal{F}_{N-1}\right]   &  =\left(  P_{N-1}^{T}b_{u}(N-1)+\sum
_{l=1}^{d}(Q_{N-1}^{l})^{T}\sigma_{x}^{l}(N-1)\right)  (u_{N-1}-u_{N-1}^{\ast
})\\
&  \  \  \  \ +\left(  P_{N-1}^{T}b_{x}(N-1)+\sum_{l=1}^{d}(Q_{N-1}^{l}%
)^{T}\sigma_{x}^{l}(N-1)\right)  \hat{X}_{N-1}%
\end{align*}
and for $k=N-2,\ldots,0$,%
\begin{align*}
&  E_{P^{\ast}}\left[  \left(  P_{k+1}^{T}b_{x}(k+1)+\sum_{l=1}^{d}%
(Q_{k+1}^{l})^{T}\sigma_{x}^{l}(k+1)+f_{x}(k+1)\right)  \hat{X}_{k+1}%
\Big{|}\mathcal{F}_{k}\right] \\
&  =\left(  P_{k}^{T}b_{u}(k)+\sum_{l=1}^{d}(Q_{k}^{l})^{T}\sigma_{u}%
^{l}(k)\right)  (u_{k}-u_{k}^{\ast})+\left(  P_{k}^{T}b_{x}(k)+\sum_{l=1}%
^{d}(Q_{k}^{l})^{T}\sigma_{x}^{l}(k)\right)  \hat{X}_{k}.
\end{align*}
From this, we can easily deduce that, for any $u\in \mathcal{U}$,
\[
E_{P^{\ast}}[\Theta^{u}]=E_{P^{\ast}}\left[  \sum_{k=0}^{N-1}H_{u}\left(
k,X_{k}^{\ast},u_{k}^{\ast},P_{k},Q_{k}\right)  (u_{k}-u_{k}^{\ast})\right]
\geq0.
\]
For each given $k\leq N-1$, taking $u_{i}=u_{i}^{\ast}$ for $i\not =k$, we get
that, for any $u\in \mathcal{U}$,%
\[
E_{P^{\ast}}\left[  H_{u}\left(  k,X_{k}^{\ast},u_{k}^{\ast},P_{k}%
,Q_{k}\right)  (u_{k}-u_{k}^{\ast})\right]  \geq0\text{.}%
\]
By Lusin's theorem, $L_{c}^{2}(\mathcal{F}_{k})$ is dense in $L_{P^{\ast}}%
^{2}(\mathcal{F}_{k})$ under the norm $(E_{P^{\ast}}\left[  |\cdot
|^{2}\right]  )^{1/2}$. Thus, for each $k\leq N-1$ and $u_{k}\in L_{P^{\ast}%
}^{2}(\mathcal{F}_{k};U_{k})$, we have%
\[
E_{P^{\ast}}\left[  H_{u}\left(  k,X_{k}^{\ast},u_{k}^{\ast},P_{k}%
,Q_{k}\right)  (u_{k}-u_{k}^{\ast})\right]  \geq0,
\]
which implies the desired result.
\end{proof}

\begin{remark}
If $U_{k}=\mathbb{R}^{m}$, then, by (\ref{e2-34}), we get $H_{u}\left(
k,X_{k}^{\ast},u_{k}^{\ast},P_{k},Q_{k}\right)  =0$, $P^{\ast}$-a.s.
\end{remark}

\subsection{Sufficient condition}

We give the following sufficient condition for optimality.

\begin{theorem}
\label{th3-7}Suppose assumptions (H1)-(H3) hold. Let $u^{\ast}=\{u_{k}^{\ast
}:k=0,\ldots,N-1\} \in \mathcal{U}$ and $P^{\ast}\in \mathcal{P}^{\ast}$ satify
that%
\[
H_{u}\left(  k,X_{k}^{\ast},u_{k}^{\ast},P_{k},Q_{k}\right)  (u-u_{k}^{\ast
})\geq0,\text{ }P^{\ast}\text{-a.s.},\text{ }\forall u\in U_{k},\text{
}k=0,\ldots,N-1,
\]
where $\{X_{k}^{\ast}:k=0,\ldots,N\}$ is the state process of (\ref{e2-1})
corresponding to $u^{\ast}$ and $\{(P_{k},Q_{k}):k=0$,$\ldots$,$N-1\}$ is the
solution of the adjoint equation (\ref{e2-30}) under $P^{\ast}$. Assume that
$H\left(  k,\cdot,\cdot,p,q\right)  $ is convex with respect to $x$, $u$ for
each $k\leq N-1$ and $\varphi$ is convex with respect to $x$. Then $u^{\ast}$
is an optimal control.
\end{theorem}

\begin{proof}
For any $u=\{u_{k}:k=0,\ldots,N-1\} \in \mathcal{U}$, let $\{X_{k}%
:k=0,\ldots,N\}$ be the state process of (\ref{e2-1}) corresponding to $u$.
Set $\xi_{k}=X_{k}-X_{k}^{\ast}$ for $k=0$,$\ldots$,$N$. Then%
\[
\left \{
\begin{array}
[c]{rl}%
\xi_{k+1}= & \displaystyle \left[  b_{x}\left(  k\right)  \xi_{k}%
+\alpha \left(  k\right)  \right]  +\sum_{l=1}^{d}\left[  \sigma_{x}^{l}\left(
k\right)  \xi_{k}+\beta^{l}\left(  k\right)  \right]  W_{k+1}^{l},\\
\xi_{0}= & 0,\text{ }k=0,\ldots,N-1,
\end{array}
\right.
\]
where%
\[
\alpha \left(  k\right)  =b(k,X_{k},u_{k})-b(k)-b_{x}\left(  k\right)  \xi
_{k},\text{ }\beta^{l}\left(  k\right)  =\sigma^{l}(k,X_{k},u_{k})-\sigma
^{l}(k)-\sigma_{x}^{l}\left(  k\right)  \xi_{k}.
\]
Since $J(u)\geq E_{P^{\ast}}\left[  \sum_{k=0}^{N-1}f(k,X_{k},u_{k}%
)+\varphi(X_{N})\right]  $, we have%
\begin{equation}
J(u)-J(u^{\ast})\geq E_{P^{\ast}}\left[  \sum_{k=0}^{N-1}(f_{x}(k)\xi
_{k}+\tilde{\alpha}(k))+\varphi_{x}(X_{N}^{\ast})\xi_{N}+\tilde{\beta}\right]
, \label{e2-35}%
\end{equation}
where%
\[
\tilde{\alpha}\left(  k\right)  =f(k,X_{k},u_{k})-f(k)-f_{x}\left(  k\right)
\xi_{k},\text{ }\tilde{\beta}=\varphi(X_{N})-\varphi(X_{N}^{\ast})-\varphi
_{x}(X_{N}^{\ast})\xi_{N}.
\]

Set $H(k)=H(k,X_{k}^{\ast},u_{k}^{\ast},P_{k},Q_{k})$ for $k=0$,$\ldots$%
,$N-1$, similar for $H_{x}(k)$ and $H_{u}(k)$. It is easy to check that%
\[
H_{x}(k)=P_{k}^{T}b_{x}(k)+\sum_{l=1}^{d}(Q_{k}^{l})^{T}\sigma_{x}%
^{l}(k)+f_{x}(k)\text{ for }k\leq N-1.
\]
Similar to the proof of Theorem \ref{th3-6}, we have%
\begin{align*}
E_{P^{\ast}}\left[  \varphi_{x}\left(  X_{N}^{\ast}\right)  \xi_{N}%
|\mathcal{F}_{N-1}\right]   &  =\left(  P_{N-1}^{T}b_{x}(N-1)+\sum_{l=1}%
^{d}(Q_{N-1}^{l})^{T}\sigma_{x}^{l}(N-1)\right)  \xi_{N-1}\\
&  \  \  \  \ +\left(  P_{N-1}^{T}\alpha(N-1)+\sum_{l=1}^{d}(Q_{N-1}^{l}%
)^{T}\beta^{l}(N-1)\right)
\end{align*}
and for $k=N-2,\ldots,0$,%
\[
E_{P^{\ast}}\left[  H_{x}(k+1)\xi_{k+1}|\mathcal{F}_{k}\right]  =\left(
P_{k}^{T}b_{x}(k)+\sum_{l=1}^{d}(Q_{k}^{l})^{T}\sigma_{x}^{l}(k)\right)
\xi_{k}+\left(  P_{k}^{T}\alpha(k)+\sum_{l=1}^{d}(Q_{k}^{l})^{T}\beta
^{l}(k)\right)  .
\]
From this, we can easily deduce that
\begin{equation}
E_{P^{\ast}}\left[  \sum_{k=0}^{N-1}f_{x}(k)\xi_{k}+\varphi_{x}(X_{N}^{\ast
})\xi_{N}\right]  =E_{P^{\ast}}\left[  \sum_{k=0}^{N-1}\left(  P_{k}^{T}%
\alpha(k)+\sum_{l=1}^{d}(Q_{k}^{l})^{T}\beta^{l}(k)\right)  \right]  .
\label{e2-36}%
\end{equation}
It is easy to verify that, for $k=0$,$\ldots$,$N-1$,%
\begin{equation}%
\begin{array}
[c]{l}%
H(k,X_{k},u_{k},P_{k},Q_{k})-H(k)-H_{x}(k)\xi_{k}\\
=\displaystyle P_{k}^{T}\alpha(k)+\sum_{l=1}^{d}(Q_{k}^{l})^{T}\beta
^{l}(k)+\tilde{\alpha}(k).
\end{array}
\label{e2-37}%
\end{equation}
Noting that $E_{P^{\ast}}\left[  H_{u}\left(  k\right)  (u_{k}-u_{k}^{\ast
})\right]  \geq0$ for $k\leq N-1$, $H\left(  k,\cdot,\cdot,p,q\right)  $ and
$\varphi(\cdot)$ are convex, we obtain by (\ref{e2-35}), (\ref{e2-36}) and
(\ref{e2-37}) that%
\begin{align*}
J(u)-J(u^{\ast})  &  \geq E_{P^{\ast}}\left[  \sum_{k=0}^{N-1}(H(k,X_{k}%
,u_{k},P_{k},Q_{k})-H(k)-H_{x}(k)\xi_{k})+\tilde{\beta}\right] \\
&  \geq E_{P^{\ast}}\left[  \sum_{k=0}^{N-1}(H(k,X_{k},u_{k},P_{k}%
,Q_{k})-H(k)-H_{x}(k)\xi_{k}-H_{u}\left(  k\right)  (u_{k}-u_{k}^{\ast
}))+\tilde{\beta}\right] \\
&  \geq0.
\end{align*}
Thus $u^{\ast}$ is an optimal control.
\end{proof}

\section{Backward algorithm and examples}

The key point to apply maximum principle (\ref{e2-34}) is to know $P^{\ast}%
\in \mathcal{P}^{\ast}$. Due to Proposition \ref{pro1-2}, we know%
\[
E_{P^{\ast}}\left[  \sum_{k=0}^{N-1}f\left(  k,X_{k}^{\ast},u_{k}^{\ast
}\right)  +\varphi \left(  X_{N}^{\ast}\right)  \Big{|}\mathcal{F}_{k}\right]
\leq \mathbb{\hat{E}}\left[  \sum_{k=0}^{N-1}f\left(  k,X_{k}^{\ast}%
,u_{k}^{\ast}\right)  +\varphi \left(  X_{N}^{\ast}\right)  \Big{|}\mathcal{F}%
_{k}\right]  ,\text{ }P^{\ast}\text{-a.s.}%
\]
Since $P^{\ast}\in \mathcal{P}^{\ast}$, we obtain%
\begin{equation}
E_{P^{\ast}}\left[  \sum_{k=0}^{N-1}f\left(  k,X_{k}^{\ast},u_{k}^{\ast
}\right)  +\varphi \left(  X_{N}^{\ast}\right)  \Big{|}\mathcal{F}_{k}\right]
=\mathbb{\hat{E}}\left[  \sum_{k=0}^{N-1}f\left(  k,X_{k}^{\ast},u_{k}^{\ast
}\right)  +\varphi \left(  X_{N}^{\ast}\right)  \Big{|}\mathcal{F}_{k}\right]
,\text{ }P^{\ast}\text{-a.s.,} \label{e3-1}%
\end{equation}
which is important to calculate $P^{\ast}\in \mathcal{P}^{\ast}$. Now, we give
the following backward algorithm to find an optimal control $u^{\ast}$:

\noindent Step 1. Calculate $P^{\ast}(\cdot|\mathcal{F}_{N-1})$ by
(\ref{e3-1}) with $k=N-1$, i.e.,%
\begin{equation}
E_{P^{\ast}}\left[  \varphi \left(  X_{N}^{\ast}\right)  |\mathcal{F}%
_{N-1}\right]  =\mathbb{\hat{E}}\left[  \varphi \left(  X_{N}^{\ast}\right)
|\mathcal{F}_{N-1}\right]  \text{ }P^{\ast}\text{-a.s.} \label{e3-2}%
\end{equation}
Specifically, under the condition $(W_{1}$,$\ldots$,$W_{N-1}$,$X_{N-1}^{\ast}%
$,$u_{N-1}^{\ast})=x$, we can find a $\theta_{x}\in \bar{\Theta}$ such that
\[
\mathbb{\hat{E}}\left[  \varphi \left(  X_{N}^{\ast}\right)  |(W_{1}%
,\ldots,W_{N-1},X_{N-1}^{\ast},u_{N-1}^{\ast})=x\right]
\]
is calculated under the distribution $F_{\theta_{x}}$ in the definition
$\mathbb{\hat{E}}\left[  \varphi \left(  X_{N}^{\ast}\right)  |\mathcal{F}%
_{N-1}\right]  $. Here $\{F_{\theta}:\theta \in \bar{\Theta}\}$ is the closure
of $\{F_{\theta}:\theta \in \Theta \}$ under the topology of weak convergence. By
(\ref{e3-2}), we obtain%
\[
P^{\ast}(\{W_{N}\in A\}|(W_{1},\ldots,W_{N-1},X_{N-1}^{\ast},u_{N-1}^{\ast
})=x)=F_{\theta_{x}}(A)\text{ for }A\in \mathcal{B}(\mathbb{R}^{d}).
\]
In particular, if $\varphi(\cdot)$, $b(N-1,\cdot)$ and $\sigma^{l}(N-1,\cdot)$
are deterministic functions, then we only need to calculate $\mathbb{\hat{E}%
}\left[  \varphi \left(  X_{N}^{\ast}\right)  |(X_{N-1}^{\ast},u_{N-1}^{\ast
})=x\right]  $.

\noindent Step 2. Calculate $P_{N-1}$ and $Q_{N-1}$ according to the adjoint
equation (\ref{e2-30}).

\noindent Step 3. Calculate $u_{N-1}^{\ast}$ by maximum principle (\ref{e2-34}).

\noindent Step 4. By repeating the above three steps, we can get $u^{\ast
}=\{u_{k}^{\ast}:k=0,\ldots,N-1\}$.

\begin{example}
Consider the case $N=4$, $d=2$, $n=m=1$ and $U_{k}=\mathbb{R}$ for $k\leq3$.
The control system is%
\[
\left \{
\begin{array}
[c]{rl}%
X_{k+1}= & (X_{k}+u_{k})+\sigma^{1}\left(  k,X_{k},u_{k}\right)  W_{k+1}%
^{1}+\sigma^{2}\left(  k,X_{k},u_{k}\right)  W_{k+1}^{2},\\
X_{0}= & 1,\text{ }k=0,\ldots,3,
\end{array}
\right.
\]
where $\sigma^{1}\left(  0,X_{0},u_{0}\right)  =\sigma^{2}\left(
0,X_{0},u_{0}\right)  =u_{0}$, $\sigma^{1}\left(  1,X_{1},u_{1}\right)
=2u_{1}$, $\sigma^{2}\left(  1,X_{1},u_{1}\right)  =u_{1}$, $\sigma^{1}\left(
2,X_{2},u_{2}\right)  =u_{2}$, $\sigma^{2}\left(  2,X_{2},u_{2}\right)
=2u_{2}$, $\sigma^{1}\left(  3,X_{3},u_{3}\right)  =\sigma^{2}\left(
3,X_{3},u_{3}\right)  =u_{3}$.

For each $\theta>0$ and $\tilde{\theta}>0$, $F_{\theta,\tilde{\theta}}$ is a
probability measure on $(\mathbb{R}^{2},\mathcal{B}(\mathbb{R}^{2}))$ defined
as%
\[
F_{\theta,\tilde{\theta}}(A)=\int_{A}\frac{1}{2\pi \theta \tilde{\theta}}%
\exp \left(  -\frac{x^{2}}{2\theta^{2}}-\frac{y^{2}}{2\tilde{\theta}^{2}%
}\right)  dxdy\text{ for }A\in \mathcal{B}(\mathbb{R}^{2}).
\]
Define%
\[
\mathbb{\hat{E}}\left[  \phi \left(  W_{k}\right)  \right]  =\max \left \{
\int_{\mathbb{R}^{2}}\phi dF_{\bar{\sigma},\underline{\sigma}},\int
_{\mathbb{R}^{2}}\phi dF_{\underline{\sigma},\bar{\sigma}}\right \}  \text{ for
}\phi \in C_{b.Lip}(\mathbb{R}^{2}),\text{ }k=1,\ldots,4,
\]
where $\underline{\sigma}=1$, $\bar{\sigma}=\sqrt{2}$. The cost functional is
defined by $J(u)=\mathbb{\hat{E}}\left[  |X_{4}|^{2}\right]  $. According to
the backward algorithm, we calculate as follows:

\noindent Step 1. Since $|X_{4}^{\ast}|^{2}=|X_{3}^{\ast}+u_{3}^{\ast}%
+u_{3}^{\ast}(W_{4}^{1}+W_{4}^{2})|^{2}$, we have%
\[
\mathbb{\hat{E}}\left[  |X_{4}^{\ast}|^{2}|\mathcal{F}_{3}\right]
=|X_{3}^{\ast}+u_{3}^{\ast}|^{2}+3|u_{3}^{\ast}|^{2}.
\]
Thus $P^{\ast}(\{W_{4}\in A\}|\mathcal{F}_{3})=F_{\bar{\sigma},\underline
{\sigma}}(A)$ or $F_{\underline{\sigma},\bar{\sigma}}(A)$.

\noindent Step 2. $P_{3}=E_{P^{\ast}}[2X_{4}^{\ast}|\mathcal{F}_{3}%
]=2(X_{3}^{\ast}+u_{3}^{\ast})$, $Q_{3}^{1}+Q_{3}^{2}=E_{P^{\ast}}%
[2X_{4}^{\ast}(W_{4}^{1}+W_{4}^{2})|\mathcal{F}_{3}]=6u_{3}^{\ast}$.

\noindent Step 3. By maximum principle (\ref{e2-34}), we have $P_{3}+Q_{3}%
^{1}+Q_{3}^{2}=0$, which yields $u_{3}^{\ast}=-X_{3}^{\ast}/4$. Thus%
\[
J(u^{\ast})=\mathbb{\hat{E}}\left[  |X_{4}^{\ast}|^{2}\right]  =\mathbb{\hat
{E}}\left[  \mathbb{\hat{E}}\left[  |X_{4}^{\ast}|^{2}|\mathcal{F}_{3}\right]
\right]  =\mathbb{\hat{E}}\left[  \left \vert X_{3}^{\ast}-\frac{1}{4}%
X_{3}^{\ast}\right \vert ^{2}+3\left \vert -\frac{1}{4}X_{3}^{\ast}\right \vert
^{2}\right]  =\frac{3}{4}\mathbb{\hat{E}}\left[  |X_{3}^{\ast}|^{2}\right]  .
\]

\noindent Step 4. Repeating the above three steps for $k=2$, $1$, $0$, we
obtain%
\[%
\begin{array}
[c]{rlrl}%
P^{\ast}(\{W_{4}\in A\}|\mathcal{F}_{3})= & F_{\bar{\sigma},\underline{\sigma
}}(A)\text{ or }F_{\underline{\sigma},\bar{\sigma}}(A),\text{ } & u_{3}^{\ast
}= & -\frac{1}{4}X_{3}^{\ast},\\
P^{\ast}(\{W_{3}\in A\}|\mathcal{F}_{2})= & F_{\underline{\sigma},\bar{\sigma
}}(A), & u_{2}^{\ast}= & -\frac{1}{10}X_{2}^{\ast},\\
P^{\ast}(\{W_{2}\in A\}|\mathcal{F}_{1})= & F_{\bar{\sigma},\underline{\sigma
}}(A), & u_{1}^{\ast}= & -\frac{1}{10}X_{1}^{\ast},\\
P^{\ast}(\{W_{1}\in A\})= & F_{\bar{\sigma},\underline{\sigma}}(A)\text{ or
}F_{\underline{\sigma},\bar{\sigma}}(A), & u_{0}^{\ast}= & -\frac{1}{4}.
\end{array}
\]
By Theorem \ref{th3-7}, we know that $u^{\ast}$ is an optimal control. The
optimal value $J(u^{\ast})=729/1600$.
\end{example}

\begin{example}
In the above example, we only change the uncertainty of the distribution of
$W_{k}$ as follows%
\[
\mathbb{\hat{E}}\left[  \phi \left(  W_{k}\right)  \right]  =\sup \left \{
\int_{\mathbb{R}^{2}}\phi dF_{\theta,\tilde{\theta}}:\theta,\tilde{\theta}%
\in \lbrack \underline{\sigma},\bar{\sigma}]\right \}  \text{ for }\phi \in
C_{b.Lip}(\mathbb{R}^{2}),\text{ }k=1,\ldots,4,
\]
the others are the same. According to the backward algorithm, we obtain%
\[%
\begin{array}
[c]{rlrl}%
P^{\ast}(\{W_{4}\in A\}|\mathcal{F}_{3})= & F_{\bar{\sigma},\bar{\sigma}%
}(A),\text{ } & u_{3}^{\ast}= & -\frac{1}{5}X_{3}^{\ast},\\
P^{\ast}(\{W_{3}\in A\}|\mathcal{F}_{2})= & F_{\bar{\sigma},\bar{\sigma}%
}(A), & u_{2}^{\ast}= & -\frac{1}{11}X_{2}^{\ast},\\
P^{\ast}(\{W_{2}\in A\}|\mathcal{F}_{1})= & F_{\bar{\sigma},\bar{\sigma}%
}(A), & u_{1}^{\ast}= & -\frac{1}{11}X_{1}^{\ast},\\
P^{\ast}(\{W_{1}\in A\})= & F_{\bar{\sigma},\bar{\sigma}}(A), & u_{0}^{\ast
}= & -\frac{1}{5}.
\end{array}
\]
The optimal value $J(u^{\ast})=64/121$.
\end{example}

\begin{example}
Consider the case $N=4$, $d=n=m=1$ and $U_{k}=\mathbb{R}$ for $k\leq3$. The
control system is%
\[
\left \{
\begin{array}
[c]{l}%
X_{1}=1+u_{0}W_{1},\\
X_{2}=\sin(\frac{\pi}{2}W_{1})X_{1}+u_{1}+\left(  X_{1}+u_{1}\right)  W_{2},\\
X_{3}=X_{2}+u_{2}W_{3},\\
X_{4}=u_{3}+X_{3}W_{4}.
\end{array}
\right.
\]
For each $\theta \in \lbrack \frac{1}{3},\frac{2}{3}]$, $F_{\theta}$ is a
probability measure on $(\mathbb{R},\mathcal{B}(\mathbb{R}))$ defined as%
\[
F_{\theta}(\{1\})=F_{\theta}(\{-1\})=\frac{\theta}{2},\text{ }F_{\theta
}(\{0\})=1-\theta.
\]
Define%
\[
\mathbb{\hat{E}}\left[  \phi \left(  W_{k}\right)  \right]  =\sup_{\theta
\in \lbrack \frac{1}{3},\frac{2}{3}]}\int_{\mathbb{R}}\phi dF_{\theta}\text{ for
}\phi \in C_{b.Lip}(\mathbb{R}),\text{ }k=1,\ldots,4,
\]
where $\int_{\mathbb{R}}\phi dF_{\theta}=\frac{\theta}{2}[\phi(1)+\phi
(-1)]+(1-\theta)\phi(0)$. It is easy to check that, for $k=1$,$\ldots$,$4$,%
\[
\mathbb{\hat{E}}\left[  W_{k}\right]  =\mathbb{\hat{E}}\left[  -W_{k}\right]
=0,\text{ \ }\mathbb{\hat{E}}\left[  W_{k}^{2}\right]  =\frac{2}{3},\text{
\ }\mathbb{\hat{E}}\left[  -W_{k}^{2}\right]  =-\frac{1}{3}.
\]
The cost functional is defined by $J(u)=\mathbb{\hat{E}}\left[  |X_{4}%
|^{2}\right]  $. According to the backward algorithm, we obtain%
\[%
\begin{array}
[c]{rlrl}%
P^{\ast}(\{W_{4}\in A\}|\mathcal{F}_{3})= & F_{2/3}(A),\text{ } & u_{3}^{\ast
}= & 0,\\
P^{\ast}(\{W_{3}\in A\}|\mathcal{F}_{2})= & F_{2/3}(A), & u_{2}^{\ast}= & 0,\\
P^{\ast}(\{W_{2}\in A\}|\mathcal{F}_{1})= & F_{2/3}(A), & u_{1}^{\ast}= &
-\frac{3}{5}\sin(\frac{\pi}{2}W_{1})X_{1}^{\ast}-\frac{2}{5}X_{1}^{\ast},\\
P^{\ast}(\{W_{1}\in A\})= & F_{1/3}(A), & u_{0}^{\ast}= & 1.
\end{array}
\]
The optimal value $J(u^{\ast})=8/45$.
\end{example}

\end{document}